\documentclass[a4paper, 11pt, english]{article}
\usepackage[utf8]{inputenc}
\usepackage[T1]{fontenc}
\usepackage{graphicx}
\usepackage{stmaryrd}
\usepackage[a4paper]{geometry}
\usepackage{amsmath,amsfonts,amssymb,amsthm,epsfig,epstopdf,url,array}
\usepackage{rotating}
\usepackage[colorlinks=true,citecolor=red,linkcolor=blue,pdfpagetransition=Blinds]{hyperref}
\usepackage{appendix}
\usepackage{cleveref}
\usepackage{nameref}
\usepackage{enumitem}
\usepackage{comment}
\Crefname{paragraph}{Section}{Sections}
\usepackage{fancyhdr}
\usepackage{xcolor}
\usepackage{mathrsfs} 

\definecolor{by}{rgb}{0.74, 0.2, 0.64}

\usepackage{fullpage}
\crefname{theo}{theorem}{theorems}

\usepackage[sort,nocompress]{cite}

\providecommand{\keywords}[1]{\noindent {\textit{Keywords:}} #1}

\theoremstyle{plain} 
\newtheorem{prop}{Proposition}[section] 
\newtheorem{theorem}[prop]{Theorem}

\newtheorem{cor}[prop]{Corollary}
\theoremstyle{definition}
\newtheorem{defi}[prop]{Definition}
\newtheorem{rmk}[prop]{Remark}

\newtheorem{assumption}[prop]{Hypothesis}
\newtheorem{example}[prop]{Example}

\makeatletter
\let\original@addcontentsline\addcontentsline
\newcommand{\dummy@addcontentsline}[3]{}
\newcommand{\DeactivateToc}{\let\addcontentsline\dummy@addcontentsline}
\newcommand{\ActivateToc}{\let\addcontentsline\original@addcontentsline}
\makeatother

\begin{document}

\title{HJB equation for
maximization of wealth under insider trading}
\author{Jorge A. Le\'on\thanks{Departamento de Control Autom\'atico, Cinvestav-IPN,
Apartado Postal 14-740, 07000 CDMX, Mexico \texttt{jleon@ctrl.cinvestav.mx}} \and Liliana Peralta\thanks{Corresponding author: Departamento de Matem\'aticas, Facultad de Ciencias, Universidad Nacional Aut\'onoma de M\'exico, Circuito Exterior, C.U., 04510 CDMX, Mexico. \texttt{lylyaanaa@ciencias.unam.mx}} \and Iv\'an Rodr\'iguez\thanks{Centro de Investigaci\'on en Matem\'aticas, Jalisco S/N, Col.Valenciana, C.P. 36000, Guanajuato, Mexico \texttt{ivan.rodriguez@cimat.mx}}}

\maketitle

\begin{abstract}
In this paper, we combine the techniques of enlargement of filtrations and stochastic
control theory to establish an extension of the verification theorem, where the coefficients of the stochastic controlled equation are adapted to the underlying filtration and the controls are adapted to a bigger filtration $\mathbf{G}$ than the one generated by the corresponding Brownian motion
$B$. 
Using the forward integral defined by Russo and Vallois \cite{RV1993}, we show that
there is a $\mathbf{G}$-adapted optimal control with respect to a certain cost functional  if and only if the Brownian motion $B$ is a $\mathbf{G}$-semimartingale.  The extended verification theorem allows us to study a financial market with an insider in order to take advantage of the extra information that the insider has from the beginning. Finally, we consider two examples  throughout the extended verification theorem.
These problems appear in financial markets with an insider.

\end{abstract}
\small
\tableofcontents
\normalsize

\section{Introduction}
The theory of enlargement of a filtration was initiated in 1976  by It\^o \cite{ito76}.  This author has pointed out that one way to extend the domain of the stochastic integral 
 (in the It\^o sense) with respect to an
$\mathbf{F}$-martingale $Y$  is to enlarge the filtration $\mathbf{F}$ to another filtration  $\mathbf{G}$ in such a way  that $Y$ remains a semimartingale with respect to the new  and bigger filtration  $\mathbf{G}$. In this way, we can now  integrate processes that are $\mathbf{G}$-adapted, which include processes that are 
not necessarily adapted to the  underlying filtration $\mathbf{F}$. In particular, It\^o \cite{ito76} shows that if
$\mathbf{G}_1$ and $\mathbf{G}_2$ are two filtrations such that
$\mathbf{G}_1\subset\mathbf{G}_2$ and $Y$ is semimartingale with respect to both filtrations, then the stochastic integrals with respect to the $\mathbf{G}_1$ and 
$\mathbf{G}_2$ semimartingale $Y$ are the same in the intersection of the domains of both integrals.  But, this problem have not been considered in \cite{ito76} when 
$\mathbf{G}_1\not\subset\mathbf{G}_2$ and  $\mathbf{G}_2\not\subset\mathbf{G}_1$.
 This problem  has been solved by Russo and Vallois \cite{RV1993} using the forward
 integral. The forward integral is a limit in probability and agrees with the It\^o integral if the integrator is a semimartingale (see Section \ref{INVTVT} and Remark \ref{rem:prop-fI}), which answer the problem that It\^o did not address. So, the forward integral is an
anticipating integral, that is, it allows us to integrate processes that are not adapted to the underlying filtration with respect to other processes that are not necessarily semimartingales, therefore the forward integral coincides with the It\^o's integral if this last one is well-defined for the filtration $\mathbf{G}$. In consequence,
the forward integral becomes an appropriate tool to deal with problems that involve
processes that are no adapted to the underlying filtration. Now, we have another
anticipating integrals such as the divergence operator in the Malliavin calculus as it is defined in Nualart \cite{Nua06}, 
or the Stratonovich integral introduced in \cite{RV1993} (see  also Le\'on\cite{Le20}).
But these integrals do not agree with the It\^o integral when we apply the enlargement of filtrations. Examples  where  
we can apply anticipating integrals, together with the Malliavin calculus,  are the study of stability of solutions to stochastic differential equations with a random
variable as initial condition (Le\'on et al. \cite{JoJoDa23}),  optimal portfolio of an investor with extra information from the beginning (see, for instance, Biagini  and {\O}ksendal \cite{biaok},  Le\'on et al. \cite{LRNNU03} and references therein, and Pikovsky and Karatzas \cite{PiKa96}), the study of stochastic differential equations driven by fractional Brownian motion, which is not a semimartingale (see, for example, Al\`os et al. \cite{ALN01}, or
Garz\'on et al. \cite{GLT17}),
 the study of short-time behaviour of the implied volatility investigated  by  Al\`os
et al. \cite{ ALV-pub}, etc.  The last problem contains only adapted processes to the underlying
filtration but employs the future volatility as a main tool, which is a process that is 
not adapted (i.e., it is an anticipating process). 

The use of the forward integral in financial markets was first introduced by Le\'on et
al. \cite{LRNNU03} to figure an optimal portfolio out of an insider to maximize the expected
logarithmic utility from terminal wealth. An insider is an investor that
possesses extra information of the development of the market from the beginning,
which is represented by a random variable $L$. In this way, we obtain an approach based on
 the Malliavin calculus  to analyse the dynamics of  the wealth equation of this insider since
 the forward integral is related to the divergence and derivative operators, as it is
shown in Nualart \cite[equality (3.14)]{Nua06}, and in Russo and Vallois 
\cite[Remark 2.5]{RV1993}.

 It is well-know that the wealth equation is a controlled stochastic differential equation. So,
the problem of calculating an optimal portfolio to maximize the utility from terminal wealth is nothing else than a problem of   stochastic control. That is, we must  compute an optimal
control that maximize/minimize a  cost functional. A main tool in stochastic control theory
is the  verification theorem, which involves an optimal control and the so called Hamilton-Jacobi-Bellman equation (for short HJB-equation). The version of the classical verification theorem considered in this paper  is the one given in the book by Korn and Korn \cite{KoKo01}. Therefore, in this verification theorem, it is natural to consider
controls that are  adapted to a bigger filtration than the underlying one  in the HJB-equation,
as it is done in Theorem \ref{the:exverthe} below. 
Thus, the first {\color{red}goal} of this paper is to study an extension of the verification theorem that
is  based on  a classical controlled stochastic differential equation and on a classical cost function, but with controls 
adapted to the filtration generated by the underlying filtration and a random variable $L$ that stands for a certain extra information of the problem (see the filtration $\mathbf{G}$ defined in \eqref{newFil}).

Since the forward integral allows us to integrate with respect to stochastic processes that are not semimartingales, we could think that we can deal with a forward controlled stochastic
differential equation driven by a process that is a martingale with respect to the underlying filtration, but with controls 
adapted to a filtration bigger than the one generated by this martingale. However, we show 
that if we can find an optimal control in this case, then the driven process is still a semimartingale
with respect to the bigger filtration. This is the second goal in this paper.

The paper is organized as follows. In Section \ref{sec:one}, we establish the framework
that we use in the remaining of this article.  Section \ref{TVT} is devoted to state the extended verification theorem.  In Section \ref{INVTVT}, we analyse an inverse type result
for the extended verification theorem. Namely, we show that if there exists an optimal control
with respect to certain cost function and certain filtration $\tilde{\mathbf{G}}$, then
the given Brownian motion is still a $\tilde{\mathbf{G}}$-semimartingale. Finally, as an example, we provide two application of our extended verification theorem, which appear in
financial markets.

\section{Statement of the problem using initial enlargement of the filtrations}\label{sec:one}
Let  $B=\{B_t:t\in[0,T]\}$ be a Brownian motion  defined on a complete probability space $(\Omega, P, \mathcal{F})$ and $\mathbf{F}=\{\mathcal{F}_t\}_{t\in[0,T]}$ the filtration generated by $B$ augmented with the null sets. It is well-known that $\{\mathcal{F}_t\}_{t\geq 0}$ satisfies the usual conditions. We know that every $\sigma$-algebra $\mathcal{F}_t$ in $\mathbf{F}$ contains the events for which is possible determine their occurrence or not only from the history of the process $B$ until time $t$. If we assume the arrival of new information from a random variable $L$ this leads up to consider a new filtration $\mathbf{G}=\{\mathcal{G}_t\}_{t\in[0,T]}$ given by
\begin{equation}\label{newFil}
\mathcal{G}_t:=\bigcap _{s>t}\left(\mathcal{F}_s \vee \sigma(L)\right),
\end{equation}
which also satisfies the usual conditions.
Under suitable assumptions on $L$
(see, for example, Yor and Mansuy \cite[Section 1.3]{Mansuy},
Le\'on \cite[Section 3]{LRNNU03}
 or Protter \cite[Section 6]{protter}), $B$ is still a special $\mathbf{G}$-semimartingale with  decomposition 
\begin{equation}
\label{ds}
B_t=\tilde{B}_t+\int_0^t\alpha_s(L) ds,\quad t\in[0,T],
\end{equation}
where $\tilde{B}$ is a $\mathbf{G}$-Brownian motion and the information drift $\alpha=\{\alpha_s(x):s\in[0,T],\ \mbox{and}\ x\in\mathbb{R}\}$
 is an $\mathbf{F}$-adapted random field such that $\alpha(L)\in L^p(0,t)$ w.p.1, for each $t\in[0,T]$ and some $p>1$. 

In the financial framework, the initial enlargement of filtrations can be interpreted in this fashion: Consider a classical financial market with one bond and one risky asset.
Then, by Karatzas \cite{ka89}, the wealth $X$ of an honest investor follows the dynamics of the It\^o's stochastic
differential equation 
\begin{equation}\label{eq0invest}
dX_t=\left(r_tX_t +(\tilde{r}_t-r_t)u_t \right)dt+u_t\sigma_tdB_t,\quad t\in[0,T].
\end{equation}
Here, $u$ stands for the amount that the investor invest in the stock (i.e., the risky asset), and the processes $r$, $\tilde{r}$ and $\sigma$ are  $\mathbf{F}$-adapted stochastic
processes that represent the rate of the bound, the rate of the stock and the volatility
of the market, respectively.  Now suppose that this investor is an insider. That is, he/she has from the beginning some extra knowledge of the future development of the market given by the 
random variable $L$. So, this insider can  use strategies of the form
$u(L)$ to invest in the stock to make profit, where $u=\{u_s(x):s\in[0,T],\ \mbox{and}\ x\in\mathbb{R}\}$ is an $\mathbf{F}$-adapted random field (see Le\'on 
et al. \cite{LRNNU03} or Navarro \cite{navarro}, and Pikovsky and Karatzas \cite{PiKa96}). In this case, from 
(\ref{ds}) and (\ref{eq0invest}), the wealth equation of the insider is
\begin{equation}\label{eq_0}
dX_t=\left(r_tX_t +(\tilde{r}_t-r_t)u_t(L) +\sigma_t\alpha_t(L)u_t(L)\right)dt+u_t(L)\sigma_td\tilde{B}_t,\quad t\in[0,T].
\end{equation}
 Actually, equations  (\ref{eq0invest}) and (\ref{eq_0}) are equivalent (i.e., they have the same solutions). Also, in this case,  we have that equation (\ref{eq0invest}) is a controlled stochastic  differential equation driven by
the $\mathbf{G}$-semimartingale $B$ that involves controls that are
$\mathbf{G}$-adapted. Hence, to take advantage of the extra information $L$,
we can figure out a $\mathbf{G}$-adapted optimal control with respect to a certain cost function, unlike the classical  stochastic control problem, where 
the controls are $\mathbf{F}$-adapted processes. This is extended as follows.

Let $U$ and $\mathcal{O}$ be a closed and an open subsets of $\mathbb{R}$, respectively. For $t_0\in [0,T)$, we will denote $Q=(t_0,T)\times \mathcal{O}$ and $\bar{Q}=[t_0,T]\times \bar{\mathcal{O}}$. Throughout this work, we will assume that the extra information is modeled by  a random variable $L$. 
Now, consider two measurable functions $b, \sigma :\bar{Q}\times U\to \mathbb{R}$ 
satisfying suitable conditions that are given in Section \ref{INVTVT} and
the controlled stochastic differential equation for the filtration $\mathbf{G}$ 
\begin{equation}\label{eq:gen}
dY_t =b(t,Y_t,u_t)dt+\sigma(t,Y_t,u_t)dB_t,\quad  t \in(0,T].
\end{equation}
Here, $u:[0,T]\times\Omega\rightarrow U$ has the form $u_s=u_s(L)$ as in equation (\ref{eq_0}). In consequence, under assumption (\ref{ds}), this last equation is also written as
\begin{equation}\label{eq:genequiv}
dY_t =\left(b(t,Y_t,u_t)+\sigma(t,Y_t,u_t)\alpha_t(L)\right)dt+\sigma(t,Y_t,u_t)d\tilde{B}_t,\quad t \in[0,T].\end{equation}
 That is, the solution $Y:\Omega\times [0,T]\to \mathcal{O}$ to these two equations
is an It\^o process adapted to the filtration $\mathbf{G}$. Therefore, $Y$ 
would  be only controlled whenever it remains in the open set $\mathcal{O}$. Thus, it is necessary to introduce the $\mathbf{G}$-stopping time
\begin{equation}\label{eq:G-stpptime}
\tau:=\inf\{s\in [t_0,T]\mid (s,Y_s) \notin Q\}.
\end{equation}
Remember that $\tau$ is a stopping time since the filtration $\mathbf{G}$ satisfies the usual conditions, as it is established in Protter \cite{protter}. Moreover, by definition, $\tau\le T$.

The main task in stochastic control consists in determining a control $u^*$ which is optimal with respect to a certain cost function. In this paper, the cost function has the form
\begin{equation}\label{fun_cost}
\mathcal{J}\left(t,x;u\right):=\mathbb{E}_{t,x}\left(\int_t^{\tau}L(s,Y_s,u_s)ds+\psi(\tau,Y_{\tau})\right),
\end{equation}
where the deterministic functions $L:Q\times U\to \mathbb{R}$ and $\psi: \bar{Q}\to \mathbb{R}$ are the initial and final cost functions, respectively. Furthermore,  the expectation $E_{t,x}$ indicates that the solution $Y$ to the controlled equation (\ref{eq:gen}) has initial condition $x$ at time $t$.
The classical tool to solve this optimization problem  is the so called Hamilton-Jacobi-Bellman equation (HJB-equation), which is related with the value function (see (\ref{Prob}) below), through the verification theorem. Consequently, in this paper we are interested in establishing an 
extension of the verification theorem that allows us to
deal with  controls  adapted to a bigger filtration than the underlying filtration, for which the Brownian motion $B$ is a semimartingale. This is done in Section
\ref{TVT}.  Conversely,  in Section \ref{INVTVT}, we show that if we can find an 
$\mathbf{G}$-adapted optimal control (with respect to a certain cost function), where
 the filtration  $\mathbf{G}$ is bigger than the one generated by $B$, then 
$B$ is an $\mathbf{G}$-semimartingale. Finally, we provide two examples where we apply our extended verification theorem in Section \ref{sec:exa}.


\section{The statement of verification theorem under enlargement of filtration}\label{TVT}

The goal of this section is to state a  verification theorem for the initial enlargement of filtrations. We first introduce the general assumptions and notation that we use throughout this section.

 $(\Omega, P, \mathcal{F})$ is a complete probability space where it is defined a Brownian motion $B=\{B_t:t\in[0,T]\}$ and $L:\Omega\rightarrow\mathbb{R}$ is a random variable such that there are a $\mathbf{G}$-Brownian motion
$\tilde{B}=\{\tilde{B}_t:t\in[0,T]\}$ and an $\mathbf{F}$-random field 
$\alpha=\{\alpha_s(x):s\in[0,T],\ \mbox{and}\ x\in\mathbb{R}\}$ satisfying
equality (\ref{ds}), for all $t\in[0,T]$, w.p.1. Here, $\mathbf{F}$ is  defined in Section \ref{sec:one} and $\mathbf{G}$ is the filtration introduced in (\ref{newFil}). In this paper, we do not necessarily have that $\mathcal{F}$ is the $\sigma$-algebra
$\mathcal{F}_T$. That is, we could have $\mathcal{F}_T\subset \mathcal{F}$. 

In this section, we deal with equation (\ref{eq:gen}). That is, the controlled stochastic
differential equation 
$$dX_t =b(t,X_t,u_t)dt+\sigma(t,X_t,u_t)dB_t.\quad t\in(t_0,T].$$
Here, $t_0\in[0,T)$, the coefficients $b, \sigma $ and the control $u$ satisfy the following hypothesis and definition, respectively. Remember that $Q$ and $U$ were introduced in Section
\ref{sec:one}.
\begin{itemize} 
\item[($\mathbf{H}$)] The coefficients $b, \sigma :Q\times U\to \mathbb{R}$ 
are measurable and satisfy the following conditions:
\begin{itemize}
\item[i)] $b(t,\cdot,u), \sigma(t,\cdot,u)\in C^{1}(\mathcal{O})$, for all 
$(t,u)\in (t_0,T)\times U$.
\item[ii)] There exists a constant $C>0$ such that, for all $(t,x,u)\in Q\times U$,
$$ |\partial_x b|\leq C, \quad |\partial_x \sigma|\leq C,
 \quad\mbox{and}\quad |b(t,x,u)|+|\sigma(t,x,u)|\leq C(1+|x|+|u|).$$
\end{itemize}\end{itemize}

Observe that equation (\ref{eq:gen}) can only have a solution $X$ up to the first time it exploits and consequently, it will be only controlled as long as it remains in the set $\mathcal{O}$. In this case, it means that equation (\ref{eq:gen})  has a solution $t\mapsto X_t$ up to either it reaches the boundary $\partial\mathcal{O}$ of the set $\mathcal{O}$, 
or $t=T$.

Now, we are ready to defined the admissible strategies.
\begin{defi}\label{def admcontrol}
Let $t_0\in[0,T)$.
A  $\mathbf{G}$-progressively measurable process $u:[t_0,T]\times\Omega\rightarrow U$
is called an admissible control for equation (\ref{eq:gen}) if
\begin{equation}\label{ad-pro}
\mathbb{E}\left(\int_t^{T}|u_s|^kds\right)<\infty,\quad \text{ for all } k\in\mathbb{N}.
\end{equation}
and, for $x\in\mathcal{O}$,  equation (\ref{eq:gen}) has a solution 
$X$ such that $X_{t_0}=x$. Moreover, we set $\mathcal{A}(t_0,x)$ as the family of
admissible controls defined on $[t_0,T]\times\Omega$.
\end{defi}

Note that if $u\in\mathcal{A}(t_0,x)$ and $x\in\mathcal{O}$, then equation (\ref{eq:gen})
has a unique solution such that $X_{t_0}=x$ because of the definition of admissible control and Hypothesis ($\mathbf{H}$).ii), which implies that the coefficients $b$ and $\sigma$ are
 Lipschitz on any interval contained in $\mathcal{O}$, uniformly on $[0,T]\times U$.

The main task in stochastic control consists in determining a control $u^*$ which is optimal with respect to a certain cost functional. For our purposes, the cost functional  has the form as in equality \eqref{fun_cost} where the deterministic functions $L:Q\times U\to \mathbb{R}$ and $\psi: \bar{Q}\to \mathbb{R}$  verify 
\begin{equation}\label{L}
|L(t,x,u)|\leq C(1+|x|^k+|u|^k)\quad\mbox{and}\and\quad |\psi(t,x)|\leq C(1+|x|^k),
\end{equation}
for some $k\in \mathbb{N}$. Remember that 
the notation $\mathbb{E}_{t,x}$ corresponds to the expectation of functionals of the 
solution  $X$ to equation \eqref{eq:gen} with an initial condition $x$ at time $t$.

Before stating the control problem for this work, we need to introduce 
some   extra definitions and conventions. 

The control problem that we consider here is to compute $u^*\in \mathcal{A}(t,x)$, which  minimizes the cost functional \eqref{fun_cost}. That is, a control $u^*$ in
$ \mathcal{A}(t,x)$ satisfying
\begin{equation}\label{Prob}
V(t,x):=\inf_{u\in\mathcal{A}(t,x)} \mathcal{J}(t,x;u)= \mathcal{J}(t,x;u^{*}).
\end{equation}
Note that the function $V:[0,T]\times\mathcal{O}\rightarrow\mathbb{R}$ describes
the evolution of the minimal costs as a function of $(t,x)$. This function is called 
the value function.

In analogy with the adapted case, where $\alpha\equiv0$, we use the convention
\begin{eqnarray}
A^{u}G(t,x):&=&\partial_t G(t,x)+\frac{1}{2}\sigma^2(t,x,u)\partial_{xx}G(t,x)\nonumber\\
&&+\left(b(t,x,u)+\alpha_t(L)\sigma(t,x,u)\right)\partial_x G(t,x),\label{eq:Au}
\end{eqnarray}
for $G\in C^{1,2}(Q)\cap C(\bar{Q})$ and  $(t,x,u)\in Q\times U$.
We observe that we need to deal with the extra term $(t,x,u)\mapsto
\alpha_t(L)\sigma(t,x,u)\partial_x G(t,x)$ since equation \eqref{eq:gen} is equivalent to
equation \eqref{eq:genequiv} due to condition \eqref{ds}. Remember that equation 
\eqref{eq:genequiv} is a controlled stochastic differential equation driven by the 
$\mathbf{G}$-Brownian motion $\tilde{B}$. So, in the remaining of this section, we assume that $\alpha(L)$ defined in \eqref{ds} belongs to $L^p([0,T]\times\Omega)$,
for some $p>1$.

Now we are in position to enunciate the main result of this section, where 
we use the $\mathbf{G}$-stopping time $\tau$ given in \eqref{eq:G-stpptime}. Note that
$\tau\equiv T$ in the case that $\mathcal{O}=\mathbb{R}$.
\begin{theorem}\label{the:exverthe}
Let Hypothesis ($\mathbf{H})$ be satisfied and let $G:Q\times \Omega \to \mathbb{R}$  be a $\mathbf{G}$-adapted random field and $\Omega_0\subset \Omega$ a set of probability $1$ such that, for all $\omega\in \Omega_0$,
$$G\in C^{1,2}(Q)\cap C(\bar{Q}),\quad
|G(t,x)|\leq K(1+|x|^m),\quad\mbox{and}\quad
|G_x(t,x)|\leq J(1+|x|^n),
$$
for some random variables $K\in L^2(\Omega)$ and $J\in L^4(\Omega)$, and  $m,n \in \mathbb{N}$. In addition, assume that $G$ is a solution of the Hamilton-Jacobi-Bellman equation
\begin{equation}\label{jhbe}
\begin{cases}
\inf_{u\in U}\left\{A^{u}G(t,x)+L(t,x,u)\right\}=0, &(t,x)\in Q, \\
\mathbb{E}_{t,x}\left(G(\tau,X_{\tau})\right)= \mathbb{E}_{t,x}\left(\psi(\tau,X_{\tau})\right),&(t,x)\in Q,
\end{cases}
\end{equation}
where $X$ is the solution of either equation \eqref{eq:gen} or equation
\eqref{eq:genequiv}. Then, if
\begin{equation}\label{norm}
\mathbb{E}_{t,x}\left(\|X\|^{\beta}\right):=\mathbb{E}_{t,x}\left(\sup_{s\in[t,\tau]}|X_s|^{\beta}\right)<\infty,\quad \mbox{for}\ \ (t,x)\in Q,
\end{equation}
with $\beta=\max(2m,k)$, where $k$ is  the exponent in \eqref{L}, we have that
\begin{description}
  \item[a)] $\mathbb{E}_{t,x}(G(t,x)) \leq \mathcal{J}(t,x,u)$, for all $(t,x)\in Q$ and $u\in \mathcal{A}(t,x)$.
  \item[b)] If for all $(t,x)\in Q$, there exits a control $u^*\in \mathcal{A}(t,x)$ such that 
\begin{equation}\label{argmin}
u^*_s\in \text{arg}\min_{u\in U}\left(A^uG(s,X_s^*)+L(s,X_s^*,u)\right),
\end{equation}
\noindent
for all $s\in[t,\tau]$, where $X^*_{s}$ is the controlled process with $X_t^*=x$
corresponding to $u^*$ via \eqref{eq:gen}, then
\begin{equation*}
\mathbb{E}_{t,x}(G(t,x))=\mathcal{J}(t,x;u^{*})=V(t,x).
\end{equation*}
\end{description}
In particular $u^{*}$ is an optimal control and $(t,x)\mapsto \mathbb{E}_{t,x}\left(G(t,x)\right)$  coincides with the value function. 
\end{theorem}
\begin{proof}

Let $(t,x)\in Q$ and $\omega_0\in\Omega_0$. Also, let $\tau$ be the 
$\mathbf{G}$-stopping time introduced in \eqref{eq:G-stpptime}.

We first assume that the open set $\mathcal{O}$ is  bounded. Then, using that $G$ 
is a solution of the HJB-equation \eqref{jhbe}, we have that, for $u\in\mathcal{A}(t,x)$
and $s\in [t,\tau)$,

\begin{equation}\label{cota_A}
0\leq A^{u_s}G(s,X_s)+L(s,X_s,u_s).
\end{equation}
On the other hand, consider a  $\mathbf{G}$-stopping time $\theta$,
such that $t\le \theta\le\tau$. Hence, by
 \eqref{eq:Au}, It\^o's formula (see \cite[Theorem 8.1, pp. 184]{kunita}) applied to $G(\theta, X_{\theta})$ and taking expectation, we obtain
\begin{align}\notag
\mathbb{E}_{t,x}&\left(G(\theta, X_{\theta})\right)\\\notag
=&\mathbb{E}_{t,x}\left(G(t,x)+\int_t^{\theta}\partial_s G(s,X_s)ds+\int_t^{\theta}\partial_x G(s,X_s)\left[b(s,X_s,u_s)+\sigma(s,X_s,u_s)\alpha_s(L)\right]ds\right.\\\notag
&+\left.\frac{1}{2}\int_t^{\theta}\partial _{xx}G(s,X_s)\sigma^2(s,X_s,u_s)ds+\int_t^{\theta}\partial_xG(s,X_s)\sigma(s,X_s,u_s)d\tilde{B}_s\right)\\\label{eq:ito_0}
=&\mathbb{E}_{t,x}\left(G(t,x)+\int_t^{\theta} A^{u_s}G(s,X_s)ds\right)+\mathbb{E}_{t,x}\left(\int_t^{\theta}\partial_x G(s,X_s)\sigma(s,X_s,u_s)d\tilde{B}_s\right).
\end{align}
Now, we claim that the expectation of the stochastic integral in equality \eqref{eq:ito_0} is equal to zero. Indeed, since $\sigma$ satisfies Hypothesis  ($\mathbf{H}$), and
using the assumption on $G_x$, we can write
\begin{align}\notag
&\mathbb{E}_{t,x}\left(\int_t^{\theta}|\partial_x G(s,X_s)\sigma(s, X_s,u_s)|^2ds\right)\leq C^2\mathbb{E}_{t,x}\left(\int_{t}^{\theta}J^2(1+|X_s|^n)^2(1+|X_s|+|u_s|)^2 ds\right)\\\notag
&= C^2\mathbb{E}_{t,x}\left(\int_t^{\theta}J^2(1+|X_s|^n)^2(1+|X_s|)^2ds\right)+2C^2\mathbb{E}_{t,x}\left(\int_t^{\theta}J^2(1+|X_s|^n)^2(1+|X_s|)|u_s|ds\right)\\\notag
&\quad+C^2\mathbb{E}_{t,x}\left(\int_t^{\theta}J^2(1+|X_s|^n)^2|u_s|^2ds\right).
\end{align}

Therefore, the fact that $\mathcal{O}$ is bounded yields that there is a constant $\tilde{C}
>0$ such that
\begin{eqnarray*}
\mathbb{E}_{t,x}\left(\int_t^{\theta}|\partial_x G(s,X_s)\sigma(s,X_s,u_s)|^2ds\right)&\le&\tilde{C}\mathbb{E}_{t,x}\left(J^2\right)+\tilde{C}\mathbb{E}_{t,x}\left(J^2
\int_t^{\theta}|u_s|ds\right)\\
&&+\tilde{C}\mathbb{E}_{t,x}\left(\int_t^{\theta}J^2|u_s|^2ds\right).
\end{eqnarray*}
Thus,  our claim is satisfied since $J\in L^4(\Omega)$ and 
condition \eqref{ad-pro}. That is,
 $$\mathbb{E}_{t,x}\left(\int_t^{\theta}|\partial_x G(s,X_s)\sigma(s.X_s,u_s)|^2ds\right)<\infty,$$ 
which implies that $$\mathbb{E}_{t,x}\left(\int_t^{\theta}\partial_x G(s,X_s)\sigma(s,X_s,u_s)d\tilde{B}_s\right)=0$$
because $\tilde{B}$ is a $\mathbf{G}$-Brownian motion.
Then, equality \eqref{eq:ito_0} becomes the inequality 
\begin{eqnarray}
\mathbb{E}_{t,x}\left(G(t,x)\right)&=&\mathbb{E}_{t,x}\left(G(\theta,X_{\theta})-\int_t^{\theta}A^{u_s}G(s,X_s)ds\right)\nonumber\\
&\leq& \mathbb{E}_{t,x}\left(G(\theta, X_{\theta})+\int_t^{\theta}L(s,X_s,u_s)ds\right),
\label{eq:intverthepaso1}
\end{eqnarray}
where to obtain the last inequality we have used \eqref{cota_A}. In particular, with $\theta=\tau$, we obtain the assertion in \textbf{a)}.

Now consider a general open set $\mathcal{O}\subset \mathbb{R}$ and see that
\eqref{eq:intverthepaso1} is also satisfied in this case. To do so, choose $N\in \mathbb{N}$ 
 such that $\frac{1}{N}<T-t$.  For $p\in\mathbb{N}$ such that $p>N$,
set
$$
\mathcal{O}_p:=\mathcal{O}\cap\left\{x\in \mathbb{R}\mid |x|<p,\; \text{dist}(x,\partial\mathcal{O})>\frac{1}{p}\right\},$$
with
$$Q_{p}:=\left[t,T-\frac{1}{p}\right)\times \mathcal{O}_p.$$
Let $\tau_p=\inf\{s\in[t,T-\frac1p)\mid (s,X_s)\notin Q_p\}$. Then,  \eqref{eq:intverthepaso1} implies
$$
\mathbb{E}_{t,x}\left(G(t,x)\right)\leq \mathbb{E}_{t,x}\left(\int_t^{\tau_p}L(s,X_s,u_s)ds+G(\tau_p,X_{\tau_p})\right),
$$
for all $(t,x)\in Q_p$ and $u\in\mathcal{A}(t,x)$. Consequently, the dominated convergence theorem,  $\tau_p\uparrow\tau$, \eqref{ad-pro},
\eqref{L},  \eqref{norm}, and the facts that $G$ is continuous in $\bar{Q}$ and
$G(t,x)\le K(1+|x|^m)$  lead to 
$$\mathbb{E}_{t,x}\left(G(t,x)\right)< \mathbb{E}_{t,x}\left(\int_t^{\tau}L(s,X_s,u_s)ds+\psi(\tau,X_{\tau})\right).$$

To finish the proof,  we now assume that, for all $(t,x)\in Q$ and $u\in\mathcal{A}(t,x)$, the following strict inequality is satisfied
$$
\mathbb{E}_{t,x}\left(G(t,x)\right)< \mathbb{E}_{t,x}\left(\int_t^{\tau}L(s,X_s,u_s)ds+\psi(\tau,X_{\tau})\right),$$
which gives
\begin{align}\label{inq:aux_3}
0< \mathbb{E}_{t,x}\left(\int_t^{\tau}(L(s,X_s,u_s)+A^{u_s}G(s,X_{s}))ds\right).
\end{align}
Indeed, from \eqref{eq:ito_0}, where we change $\theta$ by $\tau_p$,
we obtain
\begin{eqnarray*}
\mathbb{E}_{t,x}\left(G(\tau_p,X_{\tau_p})-G(t,x)\right)&=&
 \mathbb{E}_{t,x}\left(\int_t^{\tau_p}A^{u_s}G(s,X_{s})ds\right)\\
&=&\mathbb{E}_{t,x}\left(\int_t^{\tau_p}(L(s,X_s,u_s)+A^{u_s}G(s,X_{s}))ds\right.\\
&&\left.+\int_t^{\tau_p}L(s,X_s,u_s)ds\right).
\end{eqnarray*}
Thus, inequality \eqref{cota_A} and the dominated and monotone convergence theorems allow us to show that \eqref{inq:aux_3} holds.
In particular, inequality \eqref{inq:aux_3} is true for the control $u^*$ that satisfies \eqref{argmin}, namely
$$
0< \mathbb{E}_{t,x}\left(\int_t^{\tau}(L(s,X^*_s;u^*_s)+A^{u^*_s}G(s,X^*_{s}))ds\right), 
$$
which yields a contradiction since $G$ is a solution of equation \eqref{jhbe}, thus 
\begin{align*}
 \mathbb{E}_{t,x}\left(G(t,x)\right)=V(t,x)=\mathcal{J}(t,x;u^*)
\end{align*}
and the proof of case \textbf{b)} is complete. 
\end{proof}

\section{A converse-type result for the verification theorem}\label{INVTVT}

The purpose of this section is to give a converse type result of the verification theorem proved in Section \ref{TVT}. Towards this end, the main tool in this section is the forward integral with respect to the Brownian motion $B$. Remember that $\mathbf{F}$ stands for the filtration 
generated by $B$ augmented with the $P$-null sets.
\begin{defi}[Forward integral]\label{FI}
Let  $v:[0,T]\times\Omega\rightarrow\mathbb{R}$ be a $\mathcal{B}([0,T])\otimes \mathcal{F}$-measurable process with integrable trajectories. We say that $v$ is forward integrable with respect to $B$ ($v\in  \text{Dom } \delta^{-}$ for short) if
\begin{equation*}
\frac{1}{\epsilon}\int_0^{T}v_s(B_{(s+\epsilon)\wedge T}-B_s)ds,
\end{equation*}
converges in probability as $\epsilon \downarrow 0$. We denote this limit by $\int_0^Tv_sd^{-}B_s$.
\end{defi}
\begin{rmk} \label{rem:prop-fI}The Forward integral has the following two properties:
\begin{itemize}
\item[i)] Assume that $v=\{v_t:t\in[0,T]\}$ is a bounded  
$\mathcal{B}([0,T])\otimes\mathcal{F}_T$-measurable  and $\mathbf{F}$-adapted
process. Then, Russo and Vallois \cite[Proposition 1.1]{RV1993} have shown that $v\in  \text{Dom } \delta^{-}$ and
$$\int_0^Tv_sd^{-}B_s=\int_0^Tv_sdB_s,$$
where the stochastic integral in the right-hand side is in the It\^o sense.
\item[ii)] Assume that $B$ is a $\tilde{\mathbf{G}}$-semimartingale, where 
$\tilde{\mathbf{G}}$ is a bigger filtration than  $\mathbf{F}$. Let
$X$  be a  $\tilde{\mathbf{G}}$-adapted process that is integrable with respect to
the $\tilde{\mathbf{G}}$-semimartingale $B$, then $X\in  \text{Dom } \delta^{-}$ and
$$\int_0^TX_sd^{-}B_s=\int_0^TX_sdB_s,$$
where the right-hand side is the It\^o integral with respect to the $\tilde{\mathbf{G}}$-semimartingale $B$. This is also proven in   Proposition 1.1 of \cite{RV1993}.
\item[iii)] Let $v\in  \text{Dom } \delta^{-}$ and $\theta$ a random variable. Then, it is 
easy to see that  $\theta v\in  \text{Dom } \delta^{-}$ and
$$\int_0^T(\theta v_s)d^{-}B_s=\theta\int_0^Tv_sd^{-}B_s.$$
\end{itemize}
\end{rmk} 

Using Definition \ref{FI} the involve stochastic equation for the wealth process of an investor is the following controlled stochastic process (see equation \eqref{eq0invest})
\begin{equation}\label{main:eq}
X_t=x+\int_0^t \left[r_sX_s +(\tilde{r}_s-r_s)u_s\right] ds+ \int_0^t u_s\sigma_s d^{-}B_s,
\quad t\in[0,T].
\end{equation}
Here, the coefficients satisfy the following condition:

\begin{assumption}\label{A2}
 $r,\tilde{r},\sigma:[0,T]\times\Omega\rightarrow\mathbb{R}$ are 
 $\mathcal{B}([0,T])\otimes\mathcal{F}_T$-measurable and $\mathbf{F}$-adapted processes such that
\begin{enumerate}
  \item  $r$ is a bounded process such that  $(r-\tilde{r})\in L^2([0,T])$ with probability 1.
  \item $\sigma>0$ is a bounded process.
  \end{enumerate}
\end{assumption}

Throughout this section, we assume that we have a filtration $\tilde{\mathbf{G}}$ bigger
than $\mathbf{F}$. The family of admissible controls are related to this filtration. That is, in this section, the family $\mathcal{A}(t,x)$ of admissible controls is the set of 
$\tilde{\mathbf{G}}$-progressively measurable
processes $u\in L^2([0,T]\times\Omega)$, for which \eqref{main:eq} has a unique solution with $X_t=x$, for all $x\in
\mathbb{R}$. Note that, in particular, we have  $u\sigma\in\text{Dom }\delta^{-}$. We
also observe that if the filtration $\tilde{\mathbf{G}}$ agrees with the filtration 
$\mathbf{G}$ introduced in \eqref{newFil} and $B$ is the $\mathbf{G}$-semimartingale 
given in \eqref{ds}, then equation \eqref{main:eq} is nothing else than equation \eqref{eq_0} due to Remark \ref{rem:prop-fI}.ii). This is the reason why the forward integral was used for the first time in \cite{navarro} to solve problems related  to financial markets.

Remember that we are interested in the optimal control problem defined on \eqref{Prob}, where we consider the cost functional given by 
\begin{equation}\label{functional_part}
\mathcal{J}(x,t;u):=\mathbb{E}_{t,x}\left(\int_t^T au_s^2ds-\exp\left(-\int_0^Tr_sds\right)X_T(u)\right), \text{ for }a,b>0.
\end{equation}

In other words, we take the classic quadratic running cost function $L(t,x,u)=au^2$ and the final cost is $\psi(t,x)=-e^{\int_0^t-r_sds}x$, which can be interpreted as the present value of quantity $x$. 

The objective is to prove that if there exists an admissible optimal control $u^*
\in\mathcal{A}(t,x)$ for the problem given in \eqref{Prob} via the cost functional \eqref{functional_part}, thus we can conclude that the $\mathbf{F}-$Brownian motion $B$ is a semimartingale in the bigger filtration $\tilde{\mathbf{G}}$. For achieving this result we will use the following hypothesis which is inspired in \cite[Theorem 3.5]{biaok}.
\begin{assumption}\label{A1} 
\begin{enumerate}
\item For all $t\in[0,T)$ and $u\in\mathcal{A}(t,x)$, the process $(s,\omega)\mapsto u_s+\chi_{(t,t+h]}(s)\theta_0(\omega)$ belongs to $\mathcal{A}(t,x)$, where $\theta_0$ is a bounded 
$\tilde{\mathbf{G}}_t$-measurable random variable and $h>0$ is such that $t+h\leq T$.
 \item There is a constant $m$ such that $0<m\le |\sigma|$  with probability 1.
\end{enumerate}
\end{assumption}

Concerning point 1 of Hypothesis \ref{A1}, we observe the following. Consider the $\mathbf{F}$-adapted random field
\begin{eqnarray}\label{auxrandfiesinsf}
X_{\tilde{t}}(y)&=&\exp\left(\int_0^{\tilde{t}}r_sds\right)x+\exp\left(\int_0^{\tilde{t}}r_sds\right)y\int_0^{\tilde{t}}\exp\left(-\int_0^sr_\eta d\eta\right)
(\tilde{r}_s-r_s)\chi_{(t,t+h]}(s)ds\nonumber\\
&&+\exp\left(\int_0^{\tilde{t}}r_sds\right)y\int_0^{\tilde{t}}
\exp\left(-\int_0^{s}r_sds\right)\sigma_s\chi_{(t,t+h]}(s)dB_s,
\end{eqnarray}
for $\tilde{t}\in[0,T]$ and 
$y\in\mathbb{R}$. The classical It\^o formula implies that $X(y)$ is a solution to the $\mathbf{F}$-adapted
stochastic differential equation
$$
X_{\tilde{t}}(y)=x+\int_0^{\tilde{t}} r_sX_s(y) ds+y\int_0^{\tilde{t}}(\tilde{r}_s-r_s) 
\chi_{(t,t+h]}(s)ds+ y\int_0^{\tilde{t}} \sigma_s\chi_{(t,t+h]}(s) dB_s,
\quad \tilde{t}\in[0,T].$$
Therefore, Remarks \ref{rem:prop-fI}.i) and \ref{rem:prop-fI}.iii) yield
that, for a $\tilde{\mathbf{G}}$-random variable $\theta$, the $\tilde{\mathbf{G}}$-adapted process $X(\theta)$ is a solution to equation
\eqref{main:eq} with $u=\theta\chi_{(t,t+h]}$. Moreover, proceeding as in Le\'on et. al. \cite{LRNNU03}, we can show that, in this case, equation \eqref{main:eq} has a unique solution
of the form $\pi(\theta)$, where $\phi=\{\pi_s(y):(s,y)\in[0,T]\times\mathbb{R}\}$ is
an $\mathbf{F}$-adapted random field satisfying suitable conditions. We observe that we 
suppose that, in Hypothesis \ref{A1}.1, the process $(s,\omega)\mapsto\chi_{(t,t+h]}(s)\theta_0(\omega)$ is an admissible control because we do not know the form of all  the solutions to equation \eqref{main:eq}. In other words, we are assuming the uniqueness
 of the
solution to \eqref{main:eq} for controls of the form $(s,\omega)\mapsto\chi_{(t,t+h]}(s)\theta_0(\omega)$.

Now, we can prove the main result of this section.
\begin{theorem}\label{the:mainext}
Suppose that Hypotheses \ref{A2}  and \ref{A1} are satisfied and that there exists an optimal control $u^{*}\in \mathcal{A}(t,x)$ for the problem defined in \eqref{Prob} with the functional \eqref{functional_part}. Then, the $\mathbf{F}$-Brownian motion
$B$ is a  $\tilde{\mathbf{G}}$-semimartingale.
\end{theorem}
\begin{proof}
In order to simplify the notation we use the the convention
\begin{equation*}
b_t:=\int_0^tr_sds,\quad t\in[0,T].
\end{equation*}
Consider the functional $H$ defined as follows 
\begin{equation*}
H(u):=\mathbb{E}_{t,x}\left(\int_t^T au_s^2ds-e^{-b_T}X_T(u)\right),\quad
\mbox{for } u\in\mathcal{A}(t,x).
\end{equation*}
Let
 $\theta_s(\omega)=\chi_{(t,t+h]}(s)\theta_0(\omega)$ be an admissible control as in Hypothesis \ref{A1}.1 and define $F(y):=H(u^{*}+y\theta)$ for all $y\in \mathbb{R}$. Then the directional derivative of $F$ is 
\begin{align}\notag
\nabla_{\hat{y}}F&=\lim_{\varepsilon\downarrow 0}\frac{1}{\varepsilon}\mathbb{E}_{t,x}\left(\int_{t}^{T}a\left[u^{*}_s+y\theta_s+\varepsilon\hat{y}\theta_s\right]^2ds-e^{-b_T}X_T(u^{*}+y\theta+\varepsilon\hat{y}\theta)\right.\\\notag
&
\quad\left.-\int_{t}^{T}a\left[u^{*}_s+y\theta_s\right]^2ds-e^{-b_T}X_T(u^{*}+y\theta)\right)\\\notag
&=\hat{y}\mathbb{E}_{t,x}\left(\int_t^T2a(u^{*}_s+y\theta_s)\theta_sds\right)\\\notag
&\quad-\lim_{\varepsilon \downarrow 0}\mathbb{E}_{t,x}\left(e^{-b_T}\left[\frac{X_T(u^{*}+y\theta+\varepsilon \hat{y}\theta;x)-X_T(u^{*}+y\theta;x)}{\varepsilon}\right]\right)\\\label{derv}
&=\hat{y}\left[\mathbb{E}_{t,x}\left(\int_t^T2a(u^{*}_s+y\theta_s)\theta_sds\right)-\mathbb{E}_{t,x}\left(e^{-b_T}X_T(\theta;0)\right)\right],\quad \mbox{for all }\hat{y}\neq 0.
\end{align}
From the analysis of the random field \eqref{auxrandfiesinsf}, we know
\begin{align}\label{eq:ito}
e^{-b_T}&X_T(\theta;0)=\int_0^Te^{-b_s}(\tilde{r}_s-r_s)\chi_{(t,t+h]}(s)\theta_0ds+\int_0^Te^{-b_s}\sigma_s\chi_{(t,t+h]}(s)\theta_0d^{-}B_s.
\end{align}
Replacing \eqref{eq:ito} into \eqref{derv}, we get
\begin{align}\notag
\nabla_{\hat{y}}F=&\hat{y}\left[\mathbb{E}_{t,x}\left(\int_t^{t+h}2a (u^{*}_s+y\theta_0)\theta_0ds\right)\right.\\\label{deriv_direc}
&-\left.\mathbb{E}_{t,x}\left(\int_t^{t+h}e^{-b_s}(\tilde{r}_s-r_s)\theta_0ds+\int_t^{t+h}e^{-b_s}\sigma_s\theta_0d^{-}B_s\right)\right],\quad \hat{y}\neq 0.
\end{align}
 By hypothesis, the functional $H$ reaches its minimum  at $u^{*}$. Therefore,  $F$ has a minimum in $y=0$. Thus,  from equation \eqref{deriv_direc}, together with Remarks \ref{A1}.i) and \ref{A1}.iii), we obtain
\begin{equation*}\label{derif}
\mathbb{E}_{t,x}\left(\theta_0\left[\int_t^{t+h}[2au^{*}_s-e^{-b_s}(\tilde{r}_s-r_s)]ds-\int_t^{t+h}e^{-b_s}\sigma_sdB_s\right]\right)=0.
\end{equation*}
Since this equality holds for all $\tilde{\mathcal{G}}_t$-measurable random variable  
$\theta_0$,  we have established
\begin{equation}\label{N}
\mathbb{E}_{t,x}\left(\left.\int_t^{t+h}[2au^{*}_s-e^{-b_s}(\tilde{r}_s-r_s)]ds-\int_t^{t+h}e^{-b_s}\sigma_sdB_s\right|\tilde{\mathcal{G}}_t\right)=0.
\end{equation}

Now, for any admissible control $u\in \mathcal{A}(t,x)$, we denote 
\begin{equation*}
N_u(t)=\int_0^t\left[2au_s-e^{-b_s}(\tilde{r}_s-r_s)\right]ds-\int_0^te^{-b_s}\sigma_s dB_s,\quad t\in[0,T].
\end{equation*}
In consequence, from identity \eqref{N} and point 1 of Hypothesis \ref{A2}, we have
\begin{equation*}
E_{t,x}\left(\left.N_{u^{*}}(t+h)\right|\mathcal{G}_t\right)=N_{u^{*}}(t),
\end{equation*}
since $N_{u^{*}}(t)$ is $\tilde{\mathcal{G}}_t$-measurable. Thus $N_{u^{*}}(t)$ is a $\tilde{\mathbf{G}}-$martingale which implies that $R_t=\int_0^te^{-b_s}\sigma_s dB_s$ is a $\tilde{\mathbf{G}}-$semimartingale. Finally, Hypothesis \ref{A2}.2 gives 
\begin{equation*}
\int_0^te^{b_s}\sigma^{-1}_sdR_s=B_t, \quad t\in[0,T],
\end{equation*}
which is a $\tilde{\mathbf{G}}-$semimartingale, therefore, the proof is complete.
\end{proof}

\begin{cor}
Let $u\in\mathcal{A}(t,x)$. Suppose  that the process
\begin{equation*}\label{N_u}
N_u(t)=\int_0^t\left[2au_s-\exp\left(-\int_0^sr_{\tau}d\tau\right)(\tilde{r}_s-r_s)\right]ds-\int_0^t\exp\left(-\int_0^sr_{\tau}d\tau\right)\sigma_sd^{-}B_s,
\end{equation*}
is a $\tilde{\mathbf{G}}$-martingale. Then $u$ is an optimal control for the problem \eqref{Prob} with the functional \eqref{functional_part}.
\end{cor}

\begin{proof}
By the proof of Theorem \ref{the:mainext}, we have that the $\mathbf{F}$-Brownian motion
$B$ is also a $\tilde{\mathbf{G}}$-semimartingale and \eqref{deriv_direc} holds
when we write $u$ instead of $u^*$.
Moreover, by Remark \ref{rem:prop-fI}.2, we get
\begin{align}\notag
\nabla_{\hat{y}}F=&\hat{y}\left[\mathbb{E}_{t,x}\left(\int_0^T2au_s\theta_sds\right)\right.\\ \label{eq:revder0}
&-\left.\mathbb{E}_{t,x}\left(\int_0^Te^{-b_s}(\tilde{r}_s-r_s)\theta_sds+\int_0^T
e^{-b_s}\sigma_s\theta_sdB_s\right)\right]=0.
\end{align}
for  $\hat{y}\neq 0$ and $\theta\in\mathcal{A}(t,x)$ of the form
\begin{equation*}
\theta_s=\sum_{i=0}^{N-1}\theta^{i}(\omega)\chi_{(t_i,t_{i+1}]}(s),\quad 0\leq s\leq T,
\end{equation*}
where $\theta^{i}$ is a   bounded  and  $\tilde{\mathcal{G}}_{t_i}$-measurable 
random variable and $0=t_0<t_1<\cdots <t_N=T$. 
Let $\mathcal{A}_0$ be the set of such processes $\theta$. Finally, using that $\mathcal{A}_0$ is dense in the set of all the 
square-integrable and $\tilde{\mathbf{G}}$-progressively measurable processes, it is not
difficult to see that \eqref{eq:revder0} is also satisfied when
$\theta$ belongs to $\mathcal{A}(t,x)$ and, therefore, the proof is complete.
\end{proof}

\section{Application of the verification theorem under enlargement of filtrations} \label{sec:exa}
The aim of this section is to study two examples through the extended verification theorem analyzed in Section \ref{TVT} (i.e., Theorem \ref{the:exverthe}).

\begin{example}\label{exifpriv1}
Let $r$ be a positive constant and $\sigma$ an $\mathbf{F}$-adapted bounded process. Consider the controlled stochastic process $X$ given by
\begin{equation}\label{eq:priexam}
X_t=x+\int_0^trX_sds+\int_0^tu_s\sigma_sdB_s,\quad t\in[0,T].
\end{equation}
However, as we have already pointed out, if we have additional information represented by a random variable 
$L$ satisfying the conditions of Section \ref{sec:one} (i.e., the filtration
$\mathbf{G}$ in \eqref{newFil} is such that  the $\mathbf{G}$-adapted process 
$\tilde{B}$ in \eqref{ds} is a $\mathbf{G}$-Brownian motion), the equation \eqref{eq:priexam} is equivalent to the  stochastic differential equation
(driven by the $\mathbf{G}$-Brownian motion $\tilde{B}$)
 \begin{equation}\label{eq:eje1_contr}
 X_t=x+\int_0^t\left(rX_s+u_s\alpha_s(L)\sigma_s\right)ds+\int_0^tu_s\sigma_sd\tilde{B}_s.
 \end{equation} 
Our purpose is to optimize the wealth at the end time $T$ reducing the costs of the control $u$. That is, to solve the  problem
\begin{equation}\label{op_prob_eje1}
\inf_{u\in\mathcal{A}(t,x)}\mathcal{J}\left(t,x;u\right):=\inf_{u\in\mathcal{A}(t,x)}\mathbb{E}_{t,x}\left(\int_t^Tau_s^2ds-b X_T\right),\quad \mbox{with } a, b>0.
\end{equation}
So,  the corresponding HJB problem associated with \eqref{eq:eje1_contr} and
\eqref{op_prob_eje1} is to find a subset $\Omega_0\subset\Omega$ such that $P(\Omega_0)=1$ and, on $\Omega_0$, compute a solution to the HJB equation
\begin{equation}\label{jhbe_eje1}
\begin{cases}
\inf_{u\in \mathbb{R}}\left\{\partial_tG(t,x)+\frac{1}{2}u^2\sigma_t^2\partial_{xx}G(t,x)+\left(rx+u\alpha_t(L)\sigma_t\right)\partial_xG(t,x)+au^2\right\}=0 &\text{in }[0,T)\times\mathbb{R}, \\
\mathbb{E}_{t,x}\left(G(T,X_{T})\right)= -b\mathbb{E}_{t,x}\left(X_T\right).
\end{cases}
\end{equation}
Note that the argument of the infimum in equation \eqref{jhbe_eje1} is a polynomial of degree 2 on the variable $u$. Thus, using the second derivative criterion, we obtain the  optimal control 
\begin{equation}\label{optcontex1}
u^{*}(L)=-\frac{\alpha_t(L)\sigma_t\partial_xG(t, X_t)}{\sigma_t^2\partial_{xx}G(t, X_t)+2a},\quad t\in[0,T].
\end{equation}
Since $u^{*}(L)$ belongs to the argument of the infimum, then it can be replaced into \eqref{jhbe_eje1} to get the equation
\begin{equation}\label{eq:aux:eje1}
\partial_tG(t,x)+rx\partial_xG(t,x)-\frac{1}{2}\frac{\alpha_t^2(L)\sigma_t^2(\partial_xG(t,x))^2}{\sigma_t^2\partial_{xx}G(t,x)+2a}=0,\quad (t,x)\in[0,T)\times \mathbb{R}.
\end{equation}
Now, we propose the function $G(t,x)=f(t)x+g_t$, where $f$ is a $\mathcal{B}([0,T])$-measurable function and $g$ a $\mathbf{G}$-adapted process, as a candidate of the solution to equation \eqref{eq:aux:eje1}. In this manner we compute the partial derivatives of $G$ and we substitute them in \eqref{eq:aux:eje1} to get
\begin{equation*}
4a\left(xf^{\prime}(t)+g^{\prime}_t+rxf(t)\right)-\alpha^2_t(L)\sigma_t^2f^2(t)=0,\quad (t,x)\in[0,T)\times \mathbb{R}.
\end{equation*}
In consequence,
$$
f^{\prime}(t)+rf(t)=0$$
and
$$
4ag^{\prime}_t-\alpha_t^2(L)\sigma_t^2f^2(t)=0.
$$
Note that the last equation imposes that $\Omega_0=\{\omega\in\Omega:
\alpha(L)\in L^2([0,T])\}$.
Under the conditions $f(T)=-b$ and $\mathbb{E}_{t,x}(g_T)=0$, the solutions
for $f$ and $g$ are
$$
f(t)=-be^{-r(t-T)}$$
and
$$g_t=\frac{b^2}{4a}\int_0^t\sigma_s^2\alpha_s^2(L)e^{-2r(s-T)}ds-\rho_0,$$
where the constant $\rho_0$ is given by
$$
\rho_0=\frac{b^2}{4a}\mathbb{E}_{t,x}\left(\int_0^T\sigma_s^2\alpha_s^2(L)e^{-2r(s-T)}ds\right).
$$
Hence, the solution $G$ of the HJB-equation \eqref{jhbe_eje1} is 
$$
G(t,x)=-xbe^{-r(t-T)}+\frac{b^2}{4a}\int_0^t\sigma_s^2\alpha_s^2(L)e^{-2r(s-T)}ds-\rho_0.
$$
Therefore, equality \eqref{optcontex1} implies that the optimal control for the problem \eqref{op_prob_eje1} is determined by 
\begin{equation}\label{optcontex1iva}
u^{*}(L)=\frac{\alpha_t(L)\sigma_tbe^{-r(t-T)}}{2a},
\end{equation}
while the value function is
$$V(t,x)=-\mathbb{E}_{t,x}\left(xbe^{-r(t-T)}+\frac{b^2}{4a}\int_t^T\sigma_s^2\alpha_s^2(L)e^{-2r(s-T)}ds\right)
$$
due to Theorem \ref{the:exverthe}.

Finally, in order to have that that $u^*$ given in \eqref{optcontex1iva} is an admissible control, by Definition \ref{def admcontrol}, we need to verify that it belongs to $L^p([0,T]\times\Omega)$, for all $p>1$.
An example of random variable $L$ such that $\alpha(L)$ defined in \eqref{ds}
is in  $L^p([0,T]\times\Omega)$, for all $p>1$ is
$$L=\int_0^{T_1}m(s)dB_s.$$
Here $T_1>T$, $m\in L^2([0,T_1])$ and $m\neq0$, with probability 1. We can use 
Yor and Mansuy \cite[Section 1.3]{Mansuy}, Navarro \cite[Section 3]{navarro} or Le\'on et al. \cite{LRNNU03}
to see that
$$\alpha_t(x)=\frac{x-\int_0^{T_1}m(s)dB_s}{\int_t^{T_1}m(s)^2ds}\quad
t\in[0,T].$$
In consequence equality \eqref{optcontex1iva} provides an admissible control.
\end{example}

\begin{example}
Here, we consider the controlled  stochastic differential equation
$$X_t=x+\int_0^tu_sds+\int_0^tu_sdB_s,\quad t\in[0,T].$$
Note that in this case, $\mathcal{A}(t,x)$ is the family of all the $\mathbf{G}$-progressively measurable processes $u:[t_0,T]\times\Omega\rightarrow \mathbb{R}$
such that
$$
\mathbb{E}\left(\int_{t_0}^{T}|u_s|^kds\right)<\infty,\quad \text{ for all } k\in\mathbb{N}.
$$
Remember that the filtration $\mathbf{G}$ is defined in \eqref{newFil}.

The cost function is given by \eqref{op_prob_eje1} again, that is,
$$
\mathcal{J}\left(t,x;u\right):=\mathbb{E}_{t,x}\left(\int_t^Tau_s^2ds-b X_T\right),\quad \mbox{with } a, b>0.$$

We observe that in the classical theory of stochastic control (i.e, there is not extra information), an optimal control is
$$u^*\equiv\frac{b}{2a}.$$

Now, as in Example  \ref{exifpriv1}, we work with $\mathbf{G}$-progressively measurable controls. From  Theorem \ref{the:exverthe}, we must study the HJB-equation$$
\begin{cases}
\inf_{u\in \mathbb{R}}\left\{\partial_tG(t,x)+\frac{1}{2}u^2\partial_{xx}G(t,x)+u\left(\alpha_t(L)+1\right)\partial_xG(t,x)+au^2\right\}=0, &\text{in }[0,T),\times\mathbb{R}, \\
\mathbb{E}_{t,x}\left(G(T,X_{T})\right)= -b\mathbb{E}_{t,x}\left(X_T\right).
\end{cases}
$$
Thus, proceeding as in Example \ref{exifpriv1}, we propose the optimal control
\begin{equation}\label{optcon2}
u^{*}(L)=-\frac{\left(\alpha_t(L)+1\right)\partial_xG(t, X_t)}{\partial_{xx}G(t, X_t)+2a},\quad t\in[0,T].\end{equation}
Substituting this control in previous HJB-equation, we have to solve the equation
\begin{eqnarray*}
\partial_tG(t,x)-\frac{\left(\alpha_t(L)+1\right)^2(\partial_xG(t,x))^2}{2(\partial_{xx}G(t,x)+2a)}&=&0,\quad (t,x)\in[0,T)\times \mathbb{R}\\
\mathbb{E}_{t,x}\left(G(T,X_{T})\right)&=&-b\mathbb{E}_{t,x}\left(X_T\right).
\end{eqnarray*}
In order to continue with our analysis, we proceed as in Example \ref{exifpriv1} again. 
It means, we propose a function $G$ of the form
$$G(t,x)=h(t)-bx,\quad  (t,x)\in[0,T)\times \mathbb{R},$$
to show that 
$$G(t,x)=\frac{b^2}{4a}\int_0^t\left(\alpha_s(L)+1\right)^2ds-bx-\rho_0,$$
is the function that we are looking for, if $\rho_0=\frac{b^2}{4a}\mathbb{E}_{t,x}(\int_0^T\left(\alpha_s(L)+1\right)^2ds),$ 
which, together with \eqref{optcon2}, yields
$$u^{*}(L)=\frac{b\left(\alpha_t(L)+1\right)}{2a},\quad t\in[0,T].$$
As we have already pointed out, the case that $\alpha(L)\equiv0$ (i.e., there is not extra information),  we
have
$$u^*\equiv\frac{b}{2a}.$$
Now, it is easy to apply Theorem \eqref{the:exverthe} to figure out the value function.

\end{example}

\section*{Acknowledgment}
The work of L. Peralta is supported by UNAM-DGAPA-PAPIIT grants IA100324, IN102822 (Mexico).


\bibliographystyle{plain}
\small{\bibliography{biboptcont}}

\end{document}